\documentclass[12pt]{amsart}

\usepackage{tikz}
\usepackage{tikz-cd}
\usetikzlibrary{arrows.meta, positioning}
\usetikzlibrary{arrows}
\usepackage{cases}
\usepackage{latexsym}
\usepackage{amsmath}
\usepackage[arrow,matrix]{xy}
\usepackage{stmaryrd}
\usepackage{amsfonts}
\usepackage{amsmath,amssymb,amscd,bbm,amsthm,mathrsfs,dsfont}
\usepackage{hyperref}
\usepackage{cite}
\usepackage{fancyhdr}
\usepackage{amsxtra,ifthen}
\usepackage{verbatim}
\usepackage{graphics}
\usepackage{tikz}
\usepackage{tikz-cd}

\theoremstyle{plain}
\newtheorem{theorem}{Theorem}[section]

\theoremstyle{definition}
\newtheorem{definition}{Definition}[section]
\newtheorem{example}{Example}[section]

\theoremstyle{remark}

\numberwithin{equation}{section}

\begin{document}

	\thispagestyle{empty}
	
	\title[Armendariz ring property via idempotent elements]{Armendariz ring property via idempotent elements}
	\author{Asma Ali and Shafahat Hussain}
	\address{Asma Ali, Department of Mathematics, Aligarh Muslim University, Aligarh-202002, India}
	\email{asma$\_$ali2@rediffmail.com}
	\address{Shafahat Hussain, Department of Mathematics, Aligarh Muslim University, Aligarh-202002, India}
	\email{hshafahat@gmail.com}

	\maketitle
	\begin{abstract} This paper introduces the concept of e-Armendariz rings, a generalization of Armendariz rings via idempotent element $e$. A ring $R$ is called right e-Armendariz if for any polynomials $f(x), g(x) \in R[x]$, the condition $f(x)g(x)=0$ implies $a_ib_je=0$ for all coefficients $a_i$ of $f(x)$ and $b_j$ of $g(x)$. We establish that this new class of rings is not left-right symmetric and provide several illustrative examples. The main result of this paper provides a characterization that a ring $R$ is right e-Armendariz if and only if the idempotent $e$ is left semicentral and the corner ring $eRe$ is an Armendariz ring. 
		\vspace{0.3cm}
		
		
	\end{abstract}	
\section{\textbf{Introduction}}
Throughout this paper, all rings are associative with identity unless otherwise specified. For a ring $R$, we denote the set of idempotent elements by $E(R)$ and the polynomial ring over $R$ by $R[x]$.

In 1974, Armendariz \cite{Armendariz} proved that for a reduced ring $R$ (i.e., a ring with no nonzero nilpotent elements), if two polynomials $f(x) = \sum_{i=0}^{n} a_i x^i$ and $g(x) = \sum_{j=0}^{m} b_j x^j$ in $R[x]$ satisfy $f(x)g(x) = 0$, then $a_i b_j = 0$ for all $i$ and $j$. This observation motivated Rege and Chhawchharia \cite{R13} to define a ring $R$ as an \textbf{Armendariz ring} if it satisfies this property. This class of rings, which properly contains reduced rings \cite{R6}, has been a subject of extensive study. The concept has led to numerous generalizations, such as weak Armendariz rings \cite{R11}, central Armendariz rings \cite{R1}, quasi-central Armendariz rings \cite{R10}, and J-Armendariz rings \cite{R14}, among others.

In a parallel development, a recent and fruitful line of inquiry in non-commutative ring theory involves generalizing classical ring properties by making them dependent on a specific idempotent $e \in E(R)$. This approach has led to the introduction of e-symmetric rings, e-reversible rings \cite{R4}, and e-reduced rings, providing a more nuanced way to classify and understand ring structures. A common theme in this research is that these generalized ``e-properties" of $R$ can often be characterized by properties of the idempotent $e$ (such as being semicentral) and the structure of the corner ring $eRe$.

The primary goal of this paper is to synthesize these two research trends. We introduce and investigate a new class of rings called \textbf{e-Armendariz rings}, which generalizes the Armendariz property using a specific idempotent element $e$.

This paper is organized as follows. In Section 2, we provide the formal definition of right and left e-Armendariz rings. We demonstrate with examples that this property is not left-right symmetric, as shown in Example 2.1, and that the class of e-Armendariz rings is distinct from that of classical Armendariz rings. The main result of this section is Theorem 2.1, which provides a complete characterization: a ring $R$ is right e-Armendariz if and only if the idempotent $e$ is left semicentral and the corner ring $eRe$ is an Armendariz ring.

	\section{\textbf{e-Armendariz Rings}}
	\begin{definition}
		Let $R$ be a ring and $e\ne 0$ be an idempotent of $R$. Then $R$ is said to be $\textit{right e-Armendariz}$ (respectively $\textit{left e-Armendariz}$) if whenever $f(x)=\displaystyle\sum_{i=0}^{n}a_ix^i$ and $g(x)=\displaystyle\sum_{j=0}^{m}b_jx^j$ satisfy $f(x)g(x)=0$ then $a_ib_je=0$ (respectively $ea_ib_j=0).$
		\end{definition}
The following example shows that $e$-Armendariz property is not left-right symmetric.
\begin{example}
	Let $R$ be a reduced ring. Consider $S=T_2(R)$ and $e=\begin{pmatrix}
		0&0\\0&1
	\end{pmatrix}$ is an idempotent of $S$. We prove that $S$ is left $e$-Armendariz but not right $e$-Armendariz. Let $f(x)=\sum_{i=0}^nA_ix^i$ and $g(x)=\sum_{j=0}^mB_jx^j$ are two nonzero polynomials in $S[x]$ such that $f(x)g(x)=0$, where $A_i=\begin{pmatrix}
	a_i&b_i\\0&c_i
	\end{pmatrix}$ and $B_j=\begin{pmatrix}
	p_j&q_j\\0&r_j
	\end{pmatrix}$ and $a_i,b_i,c_i,p_j,q_j,r_j\in R.$ Using $f(x)g(x)=0$, we get the equations $A_0B_0=0, A_0B_1+A_1B_0=0,\cdots,A_nB_m=0.$ On simplifying $A_0B_0=0$, we get\begin{eqnarray*}
	A_0B_0 &=&
	\begin{pmatrix}
		0 & 0 \\ 0 & 0
	\end{pmatrix}\\
	\begin{pmatrix}
		a_0 & b_0 \\ 0 & c_0
	\end{pmatrix}
	\begin{pmatrix}
		p_0 & q_0 \\ 0 & r_0
	\end{pmatrix}&=&\begin{pmatrix}
	0 & 0 \\ 0 & 0
	\end{pmatrix} \\
	\begin{pmatrix}
	a_0p_0 & a_0q_0 + b_0r_0 \\ 0 & c_0r_0
\end{pmatrix}&=&\begin{pmatrix}
0 & 0 \\ 0 & 0
\end{pmatrix},
	\end{eqnarray*} which implies, \begin{equation}
	c_0r_0=0.
	\end{equation}
	Now simplify, $A_0B_1+A_1B_0=0$
\begin{eqnarray*}
	A_0B_1 + A_1B_0 &=&
	\begin{pmatrix}
		0 & 0 \\ 0 & 0
	\end{pmatrix} \\
	\begin{pmatrix}
		a_0 & b_0 \\ 0 & c_0
	\end{pmatrix}
	\begin{pmatrix}
		p_1 & q_1 \\ 0 & r_1
	\end{pmatrix}
	+ 
	\begin{pmatrix}
		a_1 & b_1 \\ 0 & c_1
	\end{pmatrix}
	\begin{pmatrix}
		p_0 & q_0 \\ 0 & r_0
	\end{pmatrix}
	&=&
	\begin{pmatrix}
		0 & 0 \\ 0 & 0
	\end{pmatrix} \\
	\begin{pmatrix}
		a_0p_1 + a_1p_0 & a_0q_1 + b_0r_1 + a_1q_0 + b_1r_0 \\ 
		0 & c_0r_1 + c_1r_0
	\end{pmatrix}
	&=&
	\begin{pmatrix}
		0 & 0 \\ 0 & 0
	\end{pmatrix},
\end{eqnarray*}
which implies,
\begin{equation}
	c_0r_1+c_1r_0=0.
\end{equation}Similarly, on solving $A_nB_m=0$, we get \begin{eqnarray*}
A_nB_m&=&\begin{pmatrix}
	0 & 0 \\ 0 & 0
\end{pmatrix}\\
\begin{pmatrix}
	a_n & b_n \\ 0 & c_n
\end{pmatrix}
\begin{pmatrix}
	p_m & q_m \\ 0 & r_m
\end{pmatrix}&=&\begin{pmatrix}
	0 & 0 \\ 0 & 0
\end{pmatrix} \\
\begin{pmatrix}
	a_np_m & a_nq_m + b_nr_m \\ 0 & c_nr_m
\end{pmatrix}&=&\begin{pmatrix}
	0 & 0 \\ 0 & 0
\end{pmatrix},
\end{eqnarray*}which implies, \begin{equation}
c_nr_m=0.
\end{equation}
Now compute the value of $eA_iB_j$,
\begin{eqnarray*}
	eA_iB_j&=&\begin{pmatrix}
		0 & 0 \\ 0 & 1
	\end{pmatrix}\begin{pmatrix}
	a_i&b_i\\0&c_i
	\end{pmatrix}\begin{pmatrix}
	p_j&q_j\\0&r_j
	\end{pmatrix}\\
	&=&\begin{pmatrix}
		0 & 0 \\ 0 & c_i
	\end{pmatrix}\begin{pmatrix}
	p_j&q_j\\0&r_j
	\end{pmatrix}\\
	&=&\begin{pmatrix}
		0 & 0 \\ 0 & c_ir_j
	\end{pmatrix}.
\end{eqnarray*}
Since $R$ is a reduced ring, $c_ir_j=0$ by \cite[Lemma 1]{Armendariz} for all $0\leq i\leq n,0\leq j \leq m$. Hence $eA_iB_j=0$ which implies $S$ is left $e$-Armendariz. Now we prove that $S$ is not right $e$-Armendariz. Take $f(x)=\begin{pmatrix}
	1&0\\0&0
\end{pmatrix}+\begin{pmatrix}
1&-1\\0&0
\end{pmatrix}x,g(x)=\begin{pmatrix}
0&0\\0&1
\end{pmatrix}+\begin{pmatrix}
0&1\\0&1
\end{pmatrix}x$. Here $f(x)g(x)=0$, but $\begin{pmatrix}
1&0\\0&0
\end{pmatrix}\begin{pmatrix}
0&1\\0&1
\end{pmatrix}e=\begin{pmatrix}
1&0\\0&0
\end{pmatrix}\begin{pmatrix}
0&1\\0&1
\end{pmatrix}\begin{pmatrix}
0&0\\0&1
\end{pmatrix}=\begin{pmatrix}
0&1\\0&0
\end{pmatrix}\ne \begin{pmatrix}
0&0\\0&0
\end{pmatrix}.$ Hence $S$ is not right $e$-Armendariz. 
\end{example}It is clear that a ring $R$ is Armendariz if and only if $R$ is $1$-Armendariz. But not all $e$-Armendariz rings are Armendariz.
\begin{example}
	Let $R$ be a reduced ring and $S=R\times T_2(R)$. We prove that $S$ is not Armendariz. Take $f(x)=\left(0,\begin{pmatrix}
		1&0\\0&0
	\end{pmatrix}\right)+\left(0,\begin{pmatrix}
	1&-1\\0&0
	\end{pmatrix}\right)x$ and $g(x)=\left(0,\begin{pmatrix}
	0&0\\0&1
	\end{pmatrix}\right)+\left(0,\begin{pmatrix}
	0&1\\0&1
	\end{pmatrix}\right)x$. Clearly, $f(x)g(x)=0$. But $\left(0,\begin{pmatrix}
	1&0\\0&0
	\end{pmatrix}\right)\left(0,\begin{pmatrix}
	0&1\\0&1
	\end{pmatrix}\right)=\left(0,\begin{pmatrix}
	0&1\\0&0
	\end{pmatrix}\right)\ne \left(0,\begin{pmatrix}
	0&0\\0&0
	\end{pmatrix}\right)$. Hence $S$ is not Armendariz. Now we prove that $S$ is $e$-Armendariz for $e=\left(1,\begin{pmatrix}
	0&0\\0&0
	\end{pmatrix}\right).$
\end{example}
\begin{proof}
	Let $f(x)=(a_0,A_0)+(a_1,A_1)x+\cdots+(a_n,A_n)x^n,g(x)=(b_0,B_0)+(b_1,B_1)x+\cdots(b_m,B_m)x^m$ where $a_i,b_j\in R$ and $A_i,B_j\in T_2(R)$. Suppose $f(x)g(x)=0$, then we get the equations $(a_0,A_0)(b_0,B_0)=0, (a_0,A_0)(b_1,B_1)+(a_1,A_1)(b_0,B_0)=0,\cdots,(a_n,A_n)(b_m,B_m)=0$. On simplifying these equations, we get $a_0b_0=0,a_0b_1+a_1b_0=0,\cdots,a_nb_m=0$. Since $R$ is a reduced ring, again using \cite[Lemma 1]{Armendariz}, we get $a_ib_j=0$. Now compute $(a_i,A_i)(b_j,B_j)e$,
	\begin{eqnarray*}
		(a_i,A_i)(b_j,B_j)e&=&(a_ib_j,A_iB_j)\left(1,\begin{pmatrix}
			0&0\\0&0
		\end{pmatrix}\right)\\
		&=&(0,A_iB_j)\left(1,\begin{pmatrix}
			0&0\\0&0
		\end{pmatrix}\right)\\
		&=&\left(0,\begin{pmatrix}
			0&0\\0&0
		\end{pmatrix}\right).
	\end{eqnarray*}
	Hence $S$ is right $e$-Armendariz. Similarly, $S$ is left $e$-Armendariz. Therefore, $S$ is $e$-Armendariz.
\end{proof}
\begin{theorem}
	The following are equivalent for a ring $R$ and $e\in E(R)$:
	\begin{enumerate}
		\item $R$ is an $e$-Armendariz ring.
		\item $eRe$ is an Armendariz ring and $e$ is left semicentral.
	\end{enumerate}
\end{theorem}
\begin{proof}
	Suppose $R$ is an $e$-Armendariz ring and let $r\in R$. Construct two polynomials $f(x)=(1-e)-(1-e)rex$ and $g(x)=e+(1-e)rex$ in $R[x]$. Clearly, $f(x)g(x)=0$, but since $R$ is an $e$-Armendariz ring, it means 
	\begin{eqnarray*}
		(1-e)(1-e)re^2&=&0\\
		\implies (1-e)re&=&0\\
		\implies re-ere&=&0\\
		\implies re&=&ere.\\	
	\end{eqnarray*}
	Hence $e$ is left semicentral. Now we prove that $eRe$ is an Armendariz ring. Suppose $f(x)=\sum_{i=0}^{m}(ea_ie)x^i$ and $g(x)=\sum_{j=0}^{n}(eb_je)x^j$ are two polynomials in $eRe[x]$ such that $f(x)g(x)=0$. Since $eRe$ is a subring of $R$, it means $f(x),g(x)\in R[x]$ and $R[x]$ is an $e$-Armendariz ring, which implies $(ea_ie)(eb_je)e=0$ and hence $(ea_ie)(eb_je)=0$ which implies $eRe$ is an $e$-Armendariz ring.\\
	Conversely, suppose that $eRe$ is an $e$-Armendariz ring and $e$ is left semicentral. Now we prove that $R$ is an $e$-Armendariz ring. Suppose $f(x)=\sum_{i=0}^{m}a_ix^i, g(x)=\sum_{j=0}^{n}b_jx^j\in R[x]$ such that $f(x)g(x)=0$. Since, $f(x)g(x)=a_0b_0+(a_0b_1+a_1b_0)x+\cdots+a_mb_nx^{m+n}=0.$ On expanding $f(x)g(x)=0$ and using that $e$ is left semicentral we get,
	\begin{eqnarray*}
	a_0b_0+(a_0b_1+a_1b_0)x+\cdots+a_mb_nx^{m+n}&=&0\\
	a_0b_0e+(a_0b_1e+a_1b_0e)x+\cdots+a_mb_nex^{m+n}&=&0\\
	a_0eb_0e+(a_0eb_1e+a_1eb_0e)x+\cdots+a_meb_nex^{m+n}&=&0\\
	ea_0eb_0e+(ea_0eb_1e+ea_1eb_0e)x+\cdots+ea_meb_nex^{m+n}&=&0\\
	(ea_0e+ea_1ex+\cdots+ea_mex^m)(eb_0e+eb_1ex+\cdots+eb_nex^n)&=&0.\\
	\end{eqnarray*}
	Hence $f(x)g(x)=(ea_0e+ea_1ex+\cdots+ea_mex^m)(eb_0e+eb_1ex+\cdots+eb_nex^n)=0$ in $eRe[x]$ but $eRe$ is an Armendariz ring, it means $(ea_ie)(eb_je)=0$ which implies $a_ieb_je=0$ which further implies $a_ib_je=0$ and hence $R$ is an $e$-Armendariz ring.
\end{proof}
In \cite{R}, N.H McCoy introduced McCoy ring as a generalization of commutative rings. A ring $R$ is said to be McCoy if whenever $f(x),g(x)\in R[x]$ satisfy $f(x)g(x)=0$ then there exist nonzero $r,s\in R$ such that $f(x)r=0$ and $sg(x)=0$. It is clear that all Armendariz rings are McCoy rings. Now we prove that all $e$-Armendariz rings are McCoy but the converse need not be true in general, hence class of $e$-Armendariz rings strictly lie between Armendariz rings and McCoy rings.
\begin{theorem}
	All $e$-Armendariz rings are McCoy.
\end{theorem}
\begin{proof}
	Suppose $R$ is an $e$-Armendariz ring and take $f(x)=\sum_{i=0}^{m}a_ix^i, g(x)=\sum_{j=0}^{n}b_jx^j\in R[x]$ such that $f(x)g(x)=0$. Since $R$ is an $e$-Armendariz ring, it implies $ea_ib_j=0$ and $a_ib_je=0$. Here $(ea_i)b_j=0$ implies $R$ is left McCoy and $a_i(b_je)=0$ impies $R$ is right McCoy. Hence $R$ is McCoy ring.
\end{proof}
The converse of the above theorem need not be true in general by the following example.
\begin{example}
	Let $\mathbb{Z}_4$ be the ring of integers modulo $4$. Consider the ring\begin{center}
		
	 $R=\left\{ \begin{pmatrix}
		a&b\\0&a
	\end{pmatrix}\big|~a,b\in \mathbb{Z}_4\right\}.$\end{center}
As $R$ is commutative then $R$ is McCoy by \cite[Theorem 2]{R}. Now we prove that $R$ is not $e$-Armendariz. Take $f(x)=g(x)=\begin{pmatrix}
	2&0\\0&2
\end{pmatrix}+\begin{pmatrix}
2&1\\0&2
\end{pmatrix}x$. Then $f(x)g(x)=0$. Now the only nonzero idempotent of $R$ is $\begin{pmatrix}
1&0\\0&1
\end{pmatrix}$. It means $\begin{pmatrix}
2&0\\0&2
\end{pmatrix}\begin{pmatrix}
2&1\\0&2
\end{pmatrix}\begin{pmatrix}
1&0\\0&1
\end{pmatrix}\ne \begin{pmatrix}
0&0\\0&0
\end{pmatrix}$. Hence $R$ is not an $e$-Armendariz ring.
\end{example}
	
\end{document}